\documentclass[a4paper,12pt]{amsart}
\usepackage[utf8]{inputenc}

\usepackage{amsmath,amsthm,amsfonts}
\usepackage{graphicx}

\newcommand{\QQ}{\mathbf{Q}}
\newcommand{\ZZ}{\mathbf{Z}}
\newcommand{\CC}{\mathbf{C}}

\renewcommand{\aa}{\mathfrak{a}}
\newcommand{\bb}{\mathfrak{b}}
\newcommand{\cc}{\mathfrak{C}}
\newcommand{\hh}{\mathfrak{h}}
\newcommand{\FF}{\mathcal{F}}
\newcommand{\im}{\mathrm{Im}}

\newcommand{\ol}[1]{\overline{#1}}

\newcommand{\psl}{\mathrm{PSL}_{2}(\ZZ)}
\newcommand{\tr}{\mathrm{tr}}

\newtheorem{thm}{Theorem}[section]
\newtheorem{theorem}[thm]{Theorem}
\newtheorem{lemma}[thm]{Lemma}
\newtheorem{corollary}[thm]{Corollary}
\newtheorem{proposition}[thm]{Proposition}

\theoremstyle{definition}

\newtheorem{definition}[thm]{Definition}
\newtheorem{example}[thm]{Example}

\numberwithin{equation}{section}

\title{M\MakeLowercase{ultivariate} L\MakeLowercase{ucas} p\MakeLowercase{olynomials} \MakeLowercase{and} i\MakeLowercase{deal} c\MakeLowercase{lasses} \MakeLowercase{in} q\MakeLowercase{uadratic} n\MakeLowercase{umber} f\MakeLowercase{ields}}
\author{Ayberk Zeytin}
\keywords{Pauli matrices, Fibonacci polynomials, Lucas polynomials, \c{c}arks, \c{c}ark hypersurfaces, indefinite binary quadratic forms, class number problems of Gau{\ss}, real quadratic number fields, narrow ideal classes}
\subjclass[2010]{11R29,11B39}

\begin{document}

\begin{abstract}
	In this work, by using Pauli matrices, we introduce four families of polynomials indexed over the positive integers. These polynomials have rational or imaginary rational coefficients. It turns out that two of these families are closely related to classical Lucas and Fibonacci polynomial sequences and hence to Lucas and Fibonacci numbers. We use one of these families to give a geometric interpretation of the 200 years old class number problems of Gau{\ss}, which is equivalent to the study of narrow ideal classes in real quadratic number fields.
\end{abstract}

\maketitle

\section{Introduction}

A number field, $K$, is a finite extension of $\QQ$. Elements of $K$ which are roots of monic polynomials with integral coefficients form a subring of $K$ called the \emph{ring of integers} of $K$ and denoted by $\ZZ_{K}$. Being an extension of $\ZZ$, $\ZZ_{K}$ shares many properties with $\ZZ$. Yet, determining for which $K$, $\ZZ_{K}$ is a unique factorization domain (which is in this case equivalent to being a principal ideal domain) is one of the most fundamental open questions of algebraic number theory. A measure for this property is the class number of $K$, denoted $h_{K}$, that is the order of the ideal class group $H(K)$, which is the multiplicative group of ideals of $\ZZ_{K}$ modulo the subgroup of principal ideals. The class number $h_{K}=1$ if and only if $\ZZ_{K}$ is a unique factorization domain. 

As $\ZZ_{K}$ is a Dedekind domain, every fractional ideal of $K$, can be generated by at most two elements. Hence one has a map from projectivized ordered pairs of elements of $\ZZ_{K}$, which is a group under ideal multiplication denoted by $H^{+}(K)$, to $H_{K}$. This map is bijective exactly when $\ZZ_{K}$ admits a unit of norm $-1$. Else, this map becomes a 2-to-1 map and the group $H^{+}(K)$ is called the \emph{narrow class group}. Analogously, the order of $H^{+}(K)$, denoted $h^{+}(K)$ is called the \emph{narrow class number}. 

Let us now restrict the extension degree to $2$, i.e. consider the quadratic case. Any such number field $K$ is equal to $\QQ(\sqrt{d})$ for some square-free integer $d$. In this case the both narrow class group and the class group is computed using the corresponding binary quadratic forms of Gau{\ss}, \cite{disquisitiones}. Whenever $d>0$, $K$ is called real quadratic and whenever $d<0$, $K$ is called imaginary. In fact, for the imaginary quadratic case Gau{\ss} has determined ``almost'' all such number fields with class number one. It turns out that $\ZZ_{K}$ is a principal ideal domain if and only if $d \in \{-3,-4,-7,-8,-11,-19,-43,-67,-163\}$. However, the question of determining real quadratic number fields of class number one is still open and is referred to as class number one problem of Gau{\ss}. It must be noted that the number of such number fields is expected to be infinite.  

In this paper, we introduce four families of multivariate polynomials named as $A_{k},B_{k},C_{k}$ and $D_{k}$. This work is focused to study the family $A_{k}$, though we point out properties of the remaining three polynomial families, too. For instance, we will see that the polynomial family $B_{k}$ is closely related to Fibonacci polynomials. The family $A_{k}$ will be called multivariate Lucas polynomials. Naming stems from the fact that if one reduces these polynomials to one variable, then the classical Lucas polynomials are obtained. A similar phenomenon occurs for the family $B_{k}$, that is they restrict to classical Fibonacci polynomials. Our main motivation for introducing such a family of polynomials is that given any real quadratic number field $K$ we use multivariate Lucas polynomials to define an affine surface, called \emph{\c{c}ark surface} of $K$, whose integral points are in one-to-one correspondence with narrow ideal classes in $K$. This allows us to access the more than 200 years old class number problems of Gau{\ss} from a completely different point of view. Indeed, \c{c}ark surfaces produce high degree projective surfaces which are conjecturally Kobayashi hyperbolic. By a conjecture of Lang they have finitely many $\QQ$-rational points. Reader is suggested to consult \cite{cark/and/lang} and references therein for further details on this point of view.

The paper is organized as follows: The next section is devoted to defining and establishing basic properties of the aforementioned polynomial sequences. In particular, multivariate Lucas and Fibonacci polynomials are defined. In the last section after a quick review of narrow ideal classes and the narrow ideal class group of a quadratic number field, we define the automorphism group of a narrow ideal class. In the real quadratic case, this group is isomorphic to $\ZZ$ generated by a hyperbolic element of $\psl$. Using this we attach an infinite bipartite ribbon graph, called \c{c}ark in \cite{UZD}, to a narrow ideal class and show how they give rise to integral points of an appropriate affine surface.

\section{Multivariable Lucas and Fibonacci polynomials}

In this section, we will introduce four families of polynomials, $A_{k},B_{k},C_{k}$ and $D_{k}$, indexed over positive integers. These polynomials have  either rational or imaginary rational coefficients. The second part is devoted to listing certain properties which will be required in upcoming sections.

\subsection{The families $A_{k}, B_{k},C_{k}$ and $D_{k}$.}
Pauli matrices, which have proven themselves to be useful tools in the context of quantum mechanics, are defined as
$$ 	\sigma_{1} = \begin{pmatrix}
                	0 & 1 \\ 1 & 0
                \end{pmatrix}, \quad 
		\sigma_{2} = \begin{pmatrix}
		             	0 & -\sqrt{-1} \\ \sqrt{-1} & 0
		             \end{pmatrix}, \quad
		\sigma_{3} = \begin{pmatrix}
		             	1 & 0 \\ 0 & -1
		             \end{pmatrix}.
$$

\noindent They are of order two and satisfy $\sigma_{1}\sigma_{2}\sigma_{3} = \sqrt{-1}$. They are traceless and their determinant is $-1$, hence their eigenvalues are $\pm1$.

Together with the identity matrix Pauli matrices form a basis for the vector space of matrices of size $2\times2$ with complex entries. In particular, for the matrix 
$M(x,y) = \begin{pmatrix}
	1+xy & x \\ y & 1
\end{pmatrix}
$
with $x$ and $y$ being complex variables, the coefficients of $I$, $\sigma_{1}$, $\sigma_{2}$ and $\sigma_{3}$ become $A = \frac{1}{2} (2 + xy)$, $B = \frac{1}{2}(x+y)$, $C = \frac{\sqrt{-1}}{2} (x-y)$ and $D = \frac{xy}{2}$, respectively. More generally, for a sequence of even number of complex numbers $x_{1}, y_{1}, \ldots, x_{k},y_{k}$ we define 
$$M(x_{1},y_{1},\ldots,x_{k},y_{k}) := M(x_{1},y_{1})\cdot\ldots\cdot M(x_{k},y_{k}).$$ 
Then there should exist polynomials, $A_{k},B_{k},C_{k}$ and $D_{k}$ in $x_{1},y_{1},\ldots,x_{k},y_{k}$ so that 

$$M = A_{k} \cdot I + B_{k} \cdot \sigma_{1} + C_{k} \cdot \sigma_{2} + D_{k} \cdot \sigma_{3}.$$ 
For instance, for $M(x_{1},y_{1},x_{2},y_{2})$, one finds immediately that 
\begin{eqnarray*}
	A_{2}(x_{1},y_{1},x_{2},y_{2}) &=& \frac{1}{2} (2 + (x_{1}+x_{2})(y_{1}+y_{2}) + x_{1}x_{2}y_{1}y_{2}) \\
	B_{2}(x_{1},y_{1},x_{2},y_{2}) &=& \frac{1}{2}( (x_{1}+x_{2})+(y_{1}+y_{2}) + x_{2}y_{1}(x_{1}+y_{2})) \\
	C_{2}(x_{1},y_{1},x_{2},y_{2}) &=& \frac{\sqrt{-1}}{2}( (x_{1}+x_{2})-(y_{1}+y_{2}) + x_{2}y_{1}(x_{1}-y_{2})) \\
	D_{2}(x_{1},y_{1},x_{2},y_{2}) &=& \frac{1}{2} ( x_{1}(y_{1}+y_{2}) - x_{2} (y_{1}-y_{2}) + x_{1}x_{2}y_{1}y_{2})
\end{eqnarray*}

The four families of polynomials satisfy the following recursive relations:

\begin{eqnarray*}
	A_{k+1}^{1,2,\ldots,k+1} &=& A_{k}^{1,2,\ldots,k}A_{1}^{k+1} + B_{k}^{1,2,\ldots,k}B_{1}^{k+1} + C_{k}^{1,2,\ldots,k}C_{1}^{k+1} + D_{k}^{1,2,\ldots,n}D_{1}^{k+1} \\
	B_{k+1}^{1,2,\ldots,k+1} &=& B_{k}^{1,2,\ldots,k}A_{1}^{k+1} + A_{k}^{1,2,\ldots,k}B_{1}^{k+1} + \sqrt{-1}(C_{k}^{1,2,\ldots,k}D_{1}^{k+1} - D_{k}^{1,2,\ldots,n}C_{1}^{k+1}) \\
	C_{k+1}^{1,2,\ldots,k+1} &=& C_{k}^{1,2,\ldots,k}A_{1}^{k+1} + A_{k}^{1,2,\ldots,k}C_{1}^{k+1} + \sqrt{-1}(D_{k}^{1,2,\ldots,k}B_{1}^{k+1} - B_{k}^{1,2,\ldots,n}D_{1}^{k+1}) \\
	D_{k+1}^{1,2,\ldots,k+1} &=& D_{k}^{1,2,\ldots,k}A_{1}^{k+1} + A_{k}^{1,2,\ldots,k}D_{1}^{k+1} + \sqrt{-1}(B_{k}^{1,2,\ldots,k}C_{1}^{k+1} - C_{k}^{1,2,\ldots,n}B_{1}^{k+1}); 
\end{eqnarray*}
where $A_{k}^{i_{1},i_{2},\ldots,i_{k}}$ stands for $A_{k}(x_{i_{1}},y_{i_{1}},\ldots,x_{i_{k}},y_{i_{k}})$. These set of relations can be obtained directly from the relations among Pauli matrices. It must be pointed out that these are not the only set of equations that defines the families $A_{k},B_{k},C_{k}$ and $D_{k}$. In fact, if we let $p(k)$ denote the number of partitions of the positive integer $k$, then there are $\frac{1}{2}p(k)$-many different such formulations. 

We refer to Table~\ref{table:ex} for the first four members of these families in which to avoid rational coefficients we multiplied each polynomial by $2$.

\begin{table}[h!]
	\centering
  \begin{tabular}{|c||p{10.7cm}|}
	  \hline\hline
	  $2A_{1}$  & $x_{1} y_{1} + 2$ \\\hline
	  $2A_{2}$  & $x_{1} x_{2} y_{1} y_{2} + x_{1} y_{1} + x_{2} y_{1} + x_{1} y_{2} + x_{2} y_{2} + 2$ \\\hline
	  $2A_{3}$  & $x_{1} x_{2} x_{3} y_{1} y_{2} y_{3} + x_{1} x_{2} y_{1} y_{2} + x_{2} x_{3} y_{1} y_{2} + x_{1} x_{2} y_{1} y_{3} + x_{1} x_{3} y_{1} y_{3} + x_{1} x_{3} y_{2} y_{3} + x_{2} x_{3} y_{2} y_{3} + x_{1} y_{1} + x_{2} y_{1} + x_{3} y_{1} + x_{1} y_{2} + x_{2} y_{2} + x_{3} y_{2} + x_{1} y_{3} + x_{2} y_{3} + x_{3} y_{3} + 2 $\\\hline
	  $2A_{4}$  & $x_{1} x_{2} x_{3} x_{4} y_{1} y_{2} y_{3} y_{4} + x_{1} x_{2} x_{3} y_{1} y_{2} y_{3} + x_{2} x_{3} x_{4} y_{1} y_{2} y_{3} + x_{1} x_{2} x_{3} y_{1} y_{2} y_{4} + x_{1} x_{2} x_{4} y_{1} y_{2} y_{4} + x_{1} x_{2} x_{4} y_{1} y_{3} y_{4} + x_{1} x_{3} x_{4} y_{1} y_{3} y_{4} + x_{1} x_{3} x_{4} y_{2} y_{3} y_{4} + x_{2} x_{3} x_{4} y_{2} y_{3} y_{4} + x_{1} x_{2} y_{1} y_{2} + x_{2} x_{3} y_{1} y_{2} + x_{2} x_{4} y_{1} y_{2} + x_{1} x_{2} y_{1} y_{3} + x_{1} x_{3} y_{1} y_{3} + x_{2} x_{4} y_{1} y_{3} + x_{3} x_{4} y_{1} y_{3} + x_{1} x_{3} y_{2} y_{3} + x_{2} x_{3} y_{2} y_{3} + x_{3} x_{4} y_{2} y_{3} + x_{1} x_{2} y_{1} y_{4} + x_{1} x_{3} y_{1} y_{4} + x_{1} x_{4} y_{1} y_{4} + x_{1} x_{3} y_{2} y_{4} + x_{2} x_{3} y_{2} y_{4} + x_{1} x_{4} y_{2} y_{4} + x_{2} x_{4} y_{2} y_{4} + x_{1} x_{4} y_{3} y_{4} + x_{2} x_{4} y_{3} y_{4} + x_{3} x_{4} y_{3} y_{4} + x_{1} y_{1} + x_{2} y_{1} + x_{3} y_{1} + x_{4} y_{1} + x_{1} y_{2} + x_{2} y_{2} + x_{3} y_{2} + x_{4} y_{2} + x_{1} y_{3} + x_{2} y_{3} + x_{3} y_{3} + x_{4} y_{3} + x_{1} y_{4} + x_{2} y_{4} + x_{3} y_{4} + x_{4} y_{4} + 2$ \\\hline\hline
	  $2B_{1}$  & $ x_{1} + y_{1}$ \\\hline
	  $2B_{2}$  & $x_{1} x_{2} y_{1} + x_{2} y_{1} y_{2} + x_{1} + x_{2} + y_{1} + y_{2}$ \\\hline
	  $2B_{3}$  & $x_{1} x_{2} x_{3} y_{1} y_{2} + x_{2} x_{3} y_{1} y_{2} y_{3} + x_{1} x_{2} y_{1} + x_{1} x_{3} y_{1} + x_{1} x_{3} y_{2} + x_{2} x_{3} y_{2} + x_{2} y_{1} y_{2} + x_{2} y_{1} y_{3} + x_{3} y_{1} y_{3} + x_{3} y_{2} y_{3} + x_{1} + x_{2} + x_{3} + y_{1} + y_{2} + y_{3}$ \\\hline
	  $2B_{4}$  & $x_{1} x_{2} x_{3} x_{4} y_{1} y_{2} y_{3} + x_{2} x_{3} x_{4} y_{1} y_{2} y_{3} y_{4} + x_{1} x_{2} x_{3} y_{1} y_{2} + x_{1} x_{2} x_{4} y_{1} y_{2} + x_{1} x_{2} x_{4} y_{1} y_{3} + x_{1} x_{3} x_{4} y_{1} y_{3} + x_{1} x_{3} x_{4} y_{2} y_{3} + x_{2} x_{3} x_{4} y_{2} y_{3} + x_{2} x_{3} y_{1} y_{2} y_{3} + x_{2} x_{3} y_{1} y_{2} y_{4} + x_{2} x_{4} y_{1} y_{2} y_{4} + x_{2} x_{4} y_{1} y_{3} y_{4} + x_{3} x_{4} y_{1} y_{3} y_{4} + x_{3} x_{4} y_{2} y_{3} y_{4} + x_{1} x_{2} y_{1} + x_{1} x_{3} y_{1} + x_{1} x_{4} y_{1} + x_{1} x_{3} y_{2} + x_{2} x_{3} y_{2} + x_{1} x_{4} y_{2} + x_{2} x_{4} y_{2} + x_{2} y_{1} y_{2} + x_{1} x_{4} y_{3} + x_{2} x_{4} y_{3} + x_{3} x_{4} y_{3} + x_{2} y_{1} y_{3} + x_{3} y_{1} y_{3} + x_{3} y_{2} y_{3} + x_{2} y_{1} y_{4} + x_{3} y_{1} y_{4} + x_{4} y_{1} y_{4} + x_{3} y_{2} y_{4} + x_{4} y_{2} y_{4} + x_{4} y_{3} y_{4} + x_{1} + x_{2} + x_{3} + x_{4} + y_{1} + y_{2} + y_{3} + y_{4}$ \\\hline\hline
	  $2C_{1}$  & $ \sqrt{-1} (x_{1} - y_{1})$ \\\hline
	  $2C_{2}$  & $\sqrt{-1} (x_{1} x_{2} y_{1} -  x_{2} y_{1} y_{2} + x_{1} + x_{2} - y_{1} - y_{2} )$ \\\hline
	  $2C_{3}$  & $\sqrt{-1} ( x_{1} x_{2} x_{3} y_{1} y_{2} -  x_{2} x_{3} y_{1} y_{2} y_{3} + x_{1} x_{2} y_{1} + x_{1} x_{3} y_{1} + x_{1} x_{3} y_{2} + x_{2} x_{3} y_{2} - x_{2} y_{1} y_{2} -  x_{2} y_{1} y_{3} - x_{3} y_{1} y_{3} - x_{3} y_{2} y_{3} + x_{1} + x_{2} + x_{3} - y_{1} - y_{2} - y_{3})$ \\\hline
	  $2C_{4}$  & $\sqrt{-1} (x_{1} x_{2} x_{3} x_{4} y_{1} y_{2} y_{3} - x_{2} x_{3} x_{4} y_{1} y_{2} y_{3} y_{4} +  x_{1} x_{2} x_{3} y_{1} y_{2} + x_{1} x_{2} x_{4} y_{1} y_{2} + x_{1} x_{2} x_{4} y_{1} y_{3} + x_{1} x_{3} x_{4} y_{1} y_{3} + x_{1} x_{3} x_{4} y_{2} y_{3} + x_{2} x_{3} x_{4} y_{2} y_{3} - x_{2} x_{3} y_{1} y_{2} y_{3} - x_{2} x_{3} y_{1} y_{2} y_{4} - x_{2} x_{4} y_{1} y_{2} y_{4} - x_{2} x_{4} y_{1} y_{3} y_{4} - x_{3} x_{4} y_{1} y_{3} y_{4} - x_{3} x_{4} y_{2} y_{3} y_{4} + x_{1} x_{2} y_{1} + x_{1} x_{3} y_{1} + x_{1} x_{4} y_{1} + x_{1} x_{3} y_{2} + x_{2} x_{3} y_{2} + x_{1} x_{4} y_{2} + x_{2} x_{4} y_{2} - x_{2} y_{1} y_{2} + x_{1} x_{4} y_{3} + x_{2} x_{4} y_{3} + x_{3} x_{4} y_{3} - x_{2} y_{1} y_{3} - x_{3} y_{1} y_{3} - x_{3} y_{2} y_{3} - x_{2} y_{1} y_{4} - x_{3} y_{1} y_{4} - x_{4} y_{1} y_{4} - x_{3} y_{2} y_{4} - x_{4} y_{2} y_{4} - x_{4} y_{3} y_{4} + x_{1} + x_{2} + x_{3} + x_{4} - y_{1} - y_{2} - y_{3} - y_{4})$ \\\hline\hline
	  $2D_{1}$  & $x_{1} y_{1}$ \\\hline
	  $2D_{2}$  & $x_{1} x_{2} y_{1} y_{2} + x_{1} y_{1} - x_{2} y_{1} + x_{1} y_{2} + x_{2} y_{2}$ \\\hline
	  $2D_{3}$  & $x_{1} x_{2} x_{3} y_{1} y_{2} y_{3} + x_{1} x_{2} y_{1} y_{2} - x_{2} x_{3} y_{1} y_{2} + x_{1} x_{2} y_{1} y_{3} + x_{1} x_{3} y_{1} y_{3} + x_{1} x_{3} y_{2} y_{3} + x_{2} x_{3} y_{2} y_{3} + x_{1} y_{1} - x_{2} y_{1} - x_{3} y_{1} + x_{1} y_{2} + x_{2} y_{2} - x_{3} y_{2} + x_{1} y_{3} + x_{2} y_{3} + x_{3} y_{3}$ \\\hline
	  $2D_{4}$  & $x_{1} x_{2} x_{3} x_{4} y_{1} y_{2} y_{3} y_{4} + x_{1} x_{2} x_{3} y_{1} y_{2} y_{3} - x_{2} x_{3} x_{4} y_{1} y_{2} y_{3} + x_{1} x_{2} x_{3} y_{1} y_{2} y_{4} + x_{1} x_{2} x_{4} y_{1} y_{2} y_{4} + x_{1} x_{2} x_{4} y_{1} y_{3} y_{4} + x_{1} x_{3} x_{4} y_{1} y_{3} y_{4} + x_{1} x_{3} x_{4} y_{2} y_{3} y_{4} + x_{2} x_{3} x_{4} y_{2} y_{3} y_{4} + x_{1} x_{2} y_{1} y_{2} - x_{2} x_{3} y_{1} y_{2} - x_{2} x_{4} y_{1} y_{2} + x_{1} x_{2} y_{1} y_{3} + x_{1} x_{3} y_{1} y_{3} - x_{2} x_{4} y_{1} y_{3} - x_{3} x_{4} y_{1} y_{3} + x_{1} x_{3} y_{2} y_{3} + x_{2} x_{3} y_{2} y_{3} - x_{3} x_{4} y_{2} y_{3} + x_{1} x_{2} y_{1} y_{4} + x_{1} x_{3} y_{1} y_{4} + x_{1} x_{4} y_{1} y_{4} + x_{1} x_{3} y_{2} y_{4} + x_{2} x_{3} y_{2} y_{4} + x_{1} x_{4} y_{2} y_{4} + x_{2} x_{4} y_{2} y_{4} + x_{1} x_{4} y_{3} y_{4} + x_{2} x_{4} y_{3} y_{4} + x_{3} x_{4} y_{3} y_{4} + x_{1} y_{1} - x_{2} y_{1} - x_{3} y_{1} - x_{4} y_{1} + x_{1} y_{2} + x_{2} y_{2} - x_{3} y_{2} - x_{4} y_{2} + x_{1} y_{3} + x_{2} y_{3} + x_{3} y_{3} - x_{4} y_{3} + x_{1} y_{4} + x_{2} y_{4} + x_{3} y_{4} + x_{4} y_{4}$\\\hline\hline
  \end{tabular}
  \caption{First few members of the families $A_{k},B_{k},C_{k}$ and $D_{k}$.}
  \label{table:ex}
\end{table}

\subsection{Properties}
\label{sec:properties}

The following is a list of properties satisfied by these polynomials:

\begin{itemize}
	\item $B_{k}, C_{k}$ and $D_{k}$ do not have any degree 0 term. That of $A_{k}$ is equal to $\frac{1}{2}\cdot 2$.
	\item For any integer $k>1$ neither of the families contain a term of the form $x_{i}^{k}$ or $y_{i}^{k}$.
	\item The polynomials $A_{k}(x,y,x,y,\ldots,x,y)$ and $D_{k}(x,y,x,y,\ldots,x,y)$ are comprised only of monomials of the form $x^{l}y^{l}$ for every $1 \leq l \leq k$. As a resuly, all monomials in $A_{k}$ and $D_{k}$ are of even degree.
	\item The polynomials $B_{k}(x,y,x,y,\ldots,x,y)$ and $C_{k}(x,y,x,y,\ldots,x,y)$ are comprised only of monomials of the form $x^{l}y^{l+1}$ for every $1 \leq l \leq k-1$ and $x^{l+1}y^{l}$ for every $1\leq l \leq k-1$. Moreover, the number of terms of the form $x^{l}y^{l+1}$ is equal to the number of terms of the form $x^{l+1}y^{l}$. As a result, all monomials in $B_{k}$ and $C_{k}$ are of odd degree.
	\item Degree of $A_{k}$ and $D_{k}$ are $2k$; whereas that  of $B_{k}$ and $C_{k}$ are $2k-1$.
	\item $2A_{k}, 2B_{k}, \frac{2}{\sqrt{-1}}C_{k}$ and $2D_{k}$ are elements of the ring $\ZZ[x_{1},y_{1},\ldots,x_{k},y_{k}]$ and are irreducible in this ring.
\end{itemize}

Proofs of the facts listed above can be obtained by induction which we leave to the reader. Proof of the following theorem exemplifies arguments involved in such proofs.

\begin{theorem}
	The polynomial $2A_{k}(x,x,\ldots,x)$ is equal to the $2k^{\mbox{th}}$ Lucas polynomial\footnote{The $k^{\mbox{th}}$ Lucas polynomial, $L_{k}(x)$, is defined as $L_{k}(x) = 2^{-k}(( x - \sqrt{x^{2} + 4})^{k} + (x + \sqrt{x^{2}+4})^{k})$. For $k \geq 1$, Lucas polynomials satisfy the recursion $L_{k+1}(x) = x L_{k}(x) + L_{k-1}(x)$, with initial conditions being $L_{0}(x) = 2$ and $L_{1}(x) = x$. The first few Lucas polynomials are $L_{2}(x) = x^{2}+2$, $L_{3}(x) = x^{3} + 3x$, $L_{4}(x) = x^{4} + 4x^{2} + 2$.}, denoted by $L_{2k}(x)$.
	\label{thm:lucas}
\end{theorem}

\begin{proof}
	We know that $L_{2k}(x) = xL_{2k-1}(x) + L_{2(k-1)}(x)$. Writing the same recursion formula for $L_{2k-1}(x)$, multiplying by $x$ and subtracting from the first, we obtain $L_{2k}(x) = (x^{2}+1)L_{2(k-1)} + xL_{2k-3}(x)$. Solving for $xL_{2k-3}(x)$ from the recursion for $L_{2(k-1)}(x)$ we finally obtain the recursion formula for the even terms in the Lucas polynomial sequence, which reads $L_{2k}(x) = (x^{2} + 2)L_{2(k-1)}(x) - L_{2(k-2)}(x)$ for $k \geq 2$. 
	On the other hand, by definition, we have $M(x,x,\ldots,x) = M(x,x)^{k}$. As noted above for $k=1$, $2A =2+x^{2}$, $B = 2x$, $C = 0$ and $2D =x^{2}$. Suppose that  $M^{k-1} =  f_{k-1}(x) + B g_{k-1}(x)\sigma_{1} + Cg_{k-1}(x)\sigma_{2} + D g_{k-1}(x)\sigma_{3}$. If we write $M^{k}$ with respect to the basis consisting of identity and the Pauli matrices then we obtain the following set of equations
	\begin{eqnarray*}
		f_{k}(x) &=& Af_{k-1}(x) + B^{2}g_{k-1}(x) + C^{2}g_{k-1}(x) + D^{2}g_{k-1}(x))\\
		g_{k}(x) &=& Ag_{k-1}(x) + f_{k-1}(x)
	\end{eqnarray*}
	We rewrite the first equality using $\det(M) = 1 = A^{2} - (B^{2}+C^{2}+D^{2})$ and get $f_{k}(x) = Af_{k-1}(x) + (A^{2}-1)g_{k-1}(x)$. This establishes the fact that there are polynomial families $f_{k}$ and $g_{k}$ indexed over the set of positive integers so that $M^{k} =  f_{k}(x) + B g_{k}(x)\sigma_{1} + Cg_{k}(x)\sigma_{2} + D g_{k}(x)\sigma_{3}$. A short algebraic manipulation on these equations gives us the recurrence relation $f_{k+1}(x) = 2Af_{k}(x) - f_{k-1}(x)$, subject to the initial conditions that $f_{0}(x) = 1$ and $f_{1}(x) = 2A = 2+x^{2}$.  
\end{proof}

Let us remark that the above method can be applied in a slightly more general setup where one obtains polynomials $f_{k}$ and $g_{k}$ of $A$ and $\det(M)$, see \cite{herpin/lucas}. One may immediately ask analogous questions for the remaining polynomial families. It is immediate to prove that $C_{k}(x,x,\ldots,x) = 0$ for any positive integer $k$. The probably more interesting result is the following result whose proof is almost identical (the only essential difference being determining the initial condition) to the proof of Theorem~\ref{thm:lucas} and therefore will be omitted.

\begin{theorem}
	The polynomial $B_{k}(x,x,\ldots,x)$ is the $2k^{\mbox{th}}$ Fibonacci polynomial\footnote{The $k^{\mbox{th}}$ Fibonacci polynomial is defined as $F_{k}(x) = 2^{-k} \frac{(x +\sqrt{x^{2}+4})^{k} - (x - \sqrt{x^{2}+4})^{k}}{\sqrt{x^{2}+4}}$. For $k \geq 1$ Fibonacci polynomials satisfy the recursion $F_{k+1}(x) = xF_{k}(x) +F_{k-1}(x)$ subject to the initial conditions $F_{0}(x) = 0$ and $F_{1}(x) = 1$. The first few Fibonacci polynomials are $F_{2}(x) = x$, $F_{3}(x) = x^{2}+1$, $F_{4}(x) = x^{3} + 2x$.}, denoted by $F_{2k}(x)$.
	\label{thm:fibonacci}
\end{theorem}

Encouraged by Theorem~\ref{thm:lucas} we make the following: 

\begin{definition}
	For $k \in \ZZ_{\geq1}$ we call the polynomial $2A_{k}$ to be the $2k^{\mbox{th}}$ \emph{multivariate Lucas polynomial} and denote it by $\mathcal{L}_{2k}$. Similarly, we define the $2k^{\mbox{th}}$ \emph{multivariate Fibonacci polynomial} as $B_{k}$ and denote it by $\mathcal{F}_{2k}$.
\end{definition}

To answer the analogous question for $D_{k}$ we note the following:

\begin{proposition}
	For any positive integer $k$ we have $\frac{B_{k}(x,x,\ldots,x)}{2x} = \frac{D_{k}(x,x,\ldots,x)}{x^{2}}$.
\end{proposition}

\begin{proof}[Sketch of proof.]
	Using the method above, one obtains a recursion formula for the polynomial $\frac{2 D_{k}(x,x,\ldots,x)}{x}$ which is exactly the same as Fibonacci polynomials with the same initial conditions.
\end{proof}

The multivariate Lucas and Fibonacci polynomial families enjoy the \emph{expected} properties of their one variable versions, which are consequences of Theorems~\ref{thm:fibonacci} and ~\ref{thm:lucas}. For instance, $\mathcal{L}_{2k}(0,0,\ldots,0) = 2$, $\mathcal{L}_{2k}(1,1,\ldots,1) = L_{2k}$; where $L_{2k}$ denote the $2k^{\mbox{th}}$ Lucas number\footnote{Lucas numbers are defined recursively as $L_{1} = 1$, $L_{2} = 3 $ and $L_{k+1} = L_{k} + L_{k-1}$.}. We also have: 
\begin{eqnarray*}
	\mathcal{L}_{2k} (x_{1},y_{1},\ldots,x_{k-1},y_{k-1},0,0) = \mathcal{L}_{2(k-1)} (x_{1},y_{1},\ldots,x_{k-1},y_{k-1}).
	\label{eq:lucas/going/down}
\end{eqnarray*}

We may then obtain the following using induction:

\begin{lemma}
	For any positive integer $l$ we have
	$$\mathcal{L}_{2(k+l)}(x_{1},y_{1},\ldots,x_{k},y_{k},\underbrace{0,0,\ldots,0,0}_{2l-\mbox{many}})= \mathcal{L}_{2k}(x_{1},y_{1},\ldots,x_{k},y_{k}).$$
	\label{thm:zeroes/in/cark/equation}
\end{lemma}

Similar properties hold also for the Fibonacci family $\mathcal{F}_{2k}$. Namely, $\mathcal{F}_{2k}(1,1,\ldots,1) = F_{2k}$; where $F_{2k}$ stands for the $2k^{\mbox{th}}$ Fibonacci number\footnote{Fibonacci numbers are defined recusively as $F_{k+1} = F_{k} + F_{k-1}$ subject to the initial conditions $F_{0} = 0$ and $F_{1} = 1$.}. We finally have 
$$\mathcal{F}_{2k}(x_{1},y_{1},\ldots,x_{k-1},y_{k-1},0,0)= \mathcal{F}_{2(k-1)}(x_{1},y_{1},\ldots,x_{k-1},y_{k-1}).$$ 
We invite the reader to discover the related phenomenon for the families $C_{k}$ and $D_{k}$.


The generator $1$ of the group $\ZZ/k\ZZ$ acts on the ordered pair $(x_{1},y_{1},\ldots,x_{k},y_{k})$ by sending it to $(x_{k},y_{k},x_{1},y_{1},\ldots,x_{k-1},y_{k-1})$. This action leaves $\mathcal{L}_{2k}$ invariant; that is

\begin{eqnarray}
	\mathcal{L}_{2k}(1 \cdot (x_{1},y_{1},\ldots,x_{k},y_{k})) = \mathcal{L}_{2k}(x_{1},y_{1},\ldots,x_{k},y_{k}).
	\label{eq:lucas/invariance}
\end{eqnarray}

Indeed, this symmetry can be seen easily by considering the action on the matrix $M(x_{1},y_{1},\ldots,x_{k},y_{k})$ and noting that the trace is invariant within a conjugacy class. This property has the consequence that for any $1\leq i \leq k$:

\begin{eqnarray*}
		\mathcal{L}_{2k}(x_{1},y_{1},\ldots,x_{i-1},y_{i-1},0,0,x_{i+1},y_{i+1},\ldots,x_{k},y_{k}) = \mathcal{L}_{2(k-1)}(x_{1},y_{1},\ldots,x_{k-1},y_{k-1}).
\end{eqnarray*}

Although $B_{k}$ and $C_{k}$ does not enjoy such a property, let us state, without proof, the following symmetry of $D_{k}$:
$$1 \cdot D_{k}(y_{1},x_{1},\ldots,y_{k},x_{k}) = D_{k}(x_{1},y_{1},\ldots,x_{k},y_{k}).$$

\section{\c{C}ark surfaces}

The main aim in this section is to define \c{c}ark surfaces using the multivariate Lucas polynomials and obtain a one-to-one correspondence between integral points of these surfaces and narrow ideal classes. Throughout $K$ stands for a real quadratic number field. We will only explain the theory of narrow ideal classes in such fields, although a much more general theory exists. Interested reader may consult \cite{lang/ant,neukirch/ant,stewart/tall/ant}.

\subsection{Narrow ideal classes.} 

For such a number field $K$, there is a square-free positive integer $d$ so that $K = \QQ(\sqrt{d})$. Since the extension degree is two its Galois group is of order $2$, and for any element $\alpha \in K$, by $\ol{\alpha}$ we denote the image of $\alpha$ under the unique non-trivial element. The ring of integers, $\ZZ_{K}$, of $K = \QQ(\sqrt{d})$ depends on $d$. More precisely, $\ZZ_{K} = 1\cdot \ZZ + \sqrt{d}\cdot \ZZ$ whenever $d \equiv 2,3 \mod 4$ and $\ZZ_{K} = 1\cdot \ZZ + \frac{1+\sqrt{d}}{2}\ZZ$ if $d \equiv 1 \mod 4$. A subset $\aa$ of $K$ is called a \emph{fractional ideal} of $\ZZ_{K}$ (or $K$) if $\aa$ is a 2 dimensional $\ZZ$-module and for which there is an integer $\xi \in \ZZ$ so that $\xi \aa \subset \ZZ_{K}$. Note that the product of two fractional ideals is again a fractional ideal. The norm of a fractional ideal $\aa$, denoted by $N(\aa)$, is defined as $\frac{1}{\xi^{2}}[\ZZ_{K}:\xi\aa]$; where $[\ZZ_{K}:\xi\aa]$ stands for the index of $\xi\aa$ in $\ZZ_{K}$. As $\xi \aa$ is an ideal in a Dedekind domain, there are at most two elements $\alpha,\beta \in \xi \aa$ so that $(\alpha,\beta) = \xi \aa$. In this case, we say that $\aa = (\frac{\alpha}{\xi},\frac{\beta}{\xi})$ and the elements are called the generators of $\aa$. 

For a fractional ideal $\aa$ generated by $\alpha,\beta \in K$, the function $f_{\aa}$ defined on the ideal $\aa$ sending any element $\nu \in \aa$ to $\frac{\nu \ol{\nu}}{N(\aa)}$ is an integral valued binary quadratic form on $\aa$. If we write $\nu = X\alpha + Y \beta$ then we have $f_{\aa}(X,Y) = aX^{2} + bXY + cY^{2}$; where $a = \frac{\alpha \ol{\alpha}}{N(\aa)}$, $b = \frac{\alpha\ol{\beta}+\ol{\alpha}\beta}{N(\aa)}$ and $c = \frac{\beta \ol{\beta}}{N(\aa)}$. One finds that the discriminant of this form, $\Delta(f_{\aa}) := b^{2} - 4ac$, is equal to $4d$ if $d \equiv 2,3 \mod 4$ and is equal to $d$ if $d \equiv 1 \mod 4$, i.e. is equal to the discriminant\footnote{The discriminant of a number field $K$ of degree $n$ is defined as the square of the determinant of the $n\times n$ matrix whose $(i,j)^{\mbox{th}}$ entry is $\sigma_{i}(\alpha_{j})$; where $\{\alpha_{1},\alpha_{2},\ldots,\alpha_{n}\}\subset \ZZ_{K}$ is a basis of $\ZZ_{K}$ and $\{\sigma_{1},\sigma_{2},\ldots,\sigma_n\}$ is the set of distinct embeddings of $K$ into $\CC$.} of $K$. Given a square-free $d$, the discriminant of the corresponding number field is called a \emph{fundamental discriminant}. As a result of the choice of $\alpha$ and $\beta$ $f_{\aa}$ is integral, that is $a,b,c \in \ZZ$ and as $d>0$ the form $f_{\aa}$ is indefinite. The binary quadratic form $f_{\aa}$ is, in addition, \emph{primitive}, that is the greatest common divisor of the coefficients $a,b$ and $c$ is 1. 

To avoid ambiguity caused by the ordering of generators, we say that a basis $(\alpha,\beta)$ of $\aa$ is oriented if $\ol{\alpha}\beta - \alpha \ol{\beta}>0$. Any element of $K$ of positive norm, say $\lambda$, maps an oriented basis to an oriented basis via sending $(\alpha,\beta)$ to $(\lambda\alpha,\lambda\beta)$. We define two fractional ideals $\aa$ and $\bb$ to be equivalent if there is an element $\lambda \in K$ of positive norm so that $\aa = \lambda\bb$. The set of equivalence classes of fractional ideals in $K$ is denoted by $H^{+}(K)$ and is a group under multiplication called the \emph{narrow class group}. An element of $H^{+}(K)$ is denoted by $[\aa] = [\alpha,\beta]$ and is called a \emph{narrow ideal class}. Note that whenever $\ZZ_{K}$ has a unit of norm $-1$, $H^{+}(K)$ turns out to be isomorphic to the classical ideal class group $H(K)$, else it is a degree two extension of the class group. In either case, we say that a narrow ideal class of $[\aa]$ in $H^{+}(K)$ lies above its ideal class in $H(K)$.

We refer to \cite{zagier/zetafunktionen/quadratische/zahlkorper} for details and proofs of the facts above.

\subsection{Automorphisms of narrow ideal classes.} 
\label{sec:automorphisms}
The group $\psl$ acts on the set of indefinite integral primitive binary quadratic forms via change of variable. The $\psl$-orbit of $f$ is denoted by $[f]$. For such a form $f(X,Y) = aX^{2} + bXY + cY^{2}$ its stabilizer is isomorphic to $\ZZ$. We call the equation $X^{2} - \Delta Z^{2} = 4$ the corresponding Pell equation; where $\Delta$ is the discriminant of $f$. The map sending an integral solution $(x,z)$ to the matrix 
$W(x,z) = \begin{pmatrix}
						\frac{x-zb}{2} & -cz\\
						az & \frac{x+zb}{2}
					\end{pmatrix}
$
gives a bijection from the set of integral solution of the corresponding Pell equation and the stabilizer of $f$. By $(x_{o},z_{o})$ we denote the solution which has the smallest positive second component among all solutions. It is called the \emph{fundamental solution}. The matrix $W_{f} = W(x_{o},z_{o})$ is called the \emph{fundamental automorphism} of $f$ which is the generator of the stabilizer of $f$, \cite[Theorem~2.5.5]{bqf/vollmer}. The other generator is $W_{f}^{-1} = W(x_{o},-z_{o})$. A direct result of this is the following:

\begin{corollary}
	If $f$ is an integral primitive indefinite binary quadratic form of discriminant $\Delta$ and $W$ is a automorphism of $f$, then $\tr(W)^{2} - 4 = z^{2}\Delta$ for some $z \in \ZZ$. 
	\label{cor:trace/plus/4}
\end{corollary}

If $f = f_{\aa}$ for some narrow ideal class $\aa$ of $K$, then an element $W$ of the stabilizer $\langle W_{f} \rangle$ is called an \emph{automorphism of} $\aa$ and the matrix $W(x_{o},z_{o})$ (with $z_{o}>0$) is called the \emph{fundamental automorphism} of $\aa$. Remark that the fundamental automorphism of all narrow ideal classes of $K$ arise from the same solution, hence the fundamental solution $(x_{o},z_{o})$ is an invariant of the number field. 

The two matrices 
$S = \begin{pmatrix}
     	0&-1\\1&0
     \end{pmatrix}
$
and 
$L = \begin{pmatrix}
     	1&-1\\1&0
     \end{pmatrix}
$
generate $\psl$ freely and therefore induce the isomorphism $\psl \cong \ZZ/2\ZZ \ast \ZZ/3\ZZ$. In particular, $W_{f}$ can be written as a word in $S, L$ and $L^{2}$. Without loss of generality, we may assume that $W_{f}$ has no cancellations. 

The action of $\psl$ on the set of narrow ideal classes of $K$ is defined as 
$\gamma \cdot (\alpha,\beta) \mapsto (p\alpha + q\beta, r\alpha+s\beta)$; where 
$\gamma = \begin{pmatrix}
				p&q\\r&s
			\end{pmatrix}$.
Note that $f_{\gamma\cdot\aa} = \gamma\cdot f_{\aa}$. The correspondence defined as $[\aa] \mapsto [f_{\aa}]$ is one-to-one, see \cite[\S10, Satz]{zagier/zetafunktionen/quadratische/zahlkorper}. This correspondence is far from being onto as there are many primitive forms of non-square-free discriminant. For instance, if $f = (a,b,c)$ is an indefinite binary quadratic form of arising from $\QQ(\sqrt{d})$ with $d$ being square-free (hence its discriminant $\Delta$ is either $d$ or $4d$ depending on the class of $d \in \ZZ/4\ZZ$) then for any prime number $p>2$ not dividing $a$, the form $(a,bp,cp^{2})$ is a primitive form of discriminant $p^{2}\Delta$.

One can also prove that the correspondence $[\aa] \mapsto [W_{f_{\aa}}]$; where $[W_{f_{\aa}}]$ stands for the conjugacy class of the stabilizer of $f_{\aa}$ is one-to-one, see \cite[Proposition~2.1]{UZD}. As above, this correspondence is not surjective even when one restricts to primitive elements(i.e. elements which are not powers of other elements). 

\subsection{\c{C}arks.}
The modular group $\psl$ acts on the upper half plane $\hh = \{z \in \CC \colon \im(z)>0\}$. An element $\gamma = \begin{pmatrix}
				p&q\\r&s
			\end{pmatrix}$
sends an element $z\in \hh$ to $\frac{pz+q}{rz+s}\in \hh$. Fixed points of the matrices $S$ and $L$ are $\sqrt{-1}$ and $\zeta_{3} = e^{2 \pi \sqrt{-1}/3}$, respectively. We mark $\sqrt{-1}$ by a $\circ$ and $\zeta_{3}$ by $\bullet$. These points are on the unit circle centered at $0 \in \CC$. The $\psl$ orbit (in $\hh$) of the part of the circle between $\circ$ and $\bullet$ is defined as the Farey tree which will be denoted by $\FF$, see Figure~\ref{fig:Farey/tree}. 

\begin{figure}[h!]
	\centering
	\includegraphics[scale=1]{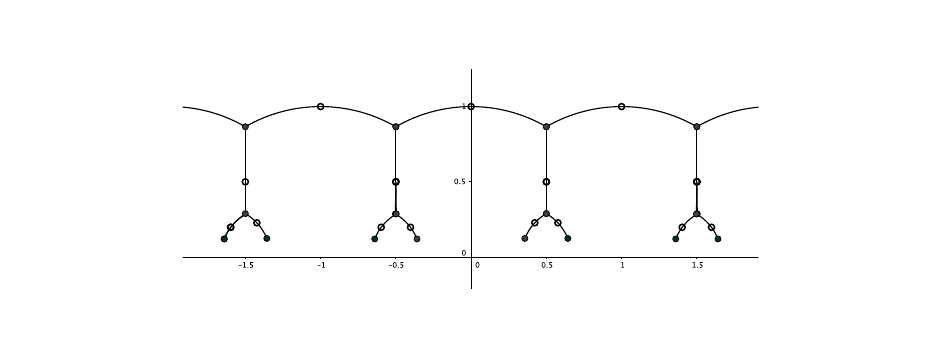}
	\caption{The Farey tree, $\FF$.}
	\label{fig:Farey/tree}
\end{figure}

The Farey tree is by construction bipartite and planar. Moreover, it admits a free action of $\psl$ in such a fashion that edges of the Farey tree can be identified with elements of $\psl$. A similar correspondence holds between vertices of type $\circ$ (resp. $\bullet$) and cosets of the torsion subgroup $\{I,S\}$ (resp. $\{I,L,L^{2}\}$). Hence, for any subgroup $\Gamma$ of $\psl$, one may talk about the quotient graph, $\Gamma \backslash \FF$, which is again bipartite but not necessarily planar as a ribbon graph. In such a graph, every vertex of type $\bullet$ is of order $1$ or $3$ and every vertex of type $\circ$ is of order $1$ or $2$. In particular, the full quotient, $\psl \backslash \hh$ is called the modular orbifold. Let us remark that the covering category consisting of \'{e}tale covers of the modular orbifold is so rich that the whole absolute Galois group can be recovered from it, see \cite{panorama}. The quotient $\psl \backslash \FF$ has only two vertices, one is $\circ$ and the other is $\bullet$ with a single edge joining the two. 

The conjugation action of $\psl$ on its subgroups is equivalent to the translation action of $\psl$ on the set of edges of the corresponding graph, see \cite[Theorem~2.2]{UZD}. Let us make the following:

\begin{definition}
	Let $\aa$ be a narrow ideal class in $K$. Then the graph $\langle W_{f_{\aa}}\rangle \backslash \FF$ is called the \emph{\c{c}ark} corresponding to $\aa$. This graph is denoted by $\cc_{\aa}$.
\end{definition}

Similar to other correspondences stated, this is again one to one but far from being surjective, as there are \c{c}arks which come from binary quadratic forms of non-square-free discriminant. Nevertheless, \c{c}arks that come from a narrow ideal class inherit all invariants of ideal classes and corresponding binary quadratic forms, e.g. discriminants, traces, etc. The graph $\cc_{\aa}$ is planar, can be embedded in an annulus conformally (\cite[\S3.2]{UZD}) and has a unique cycle called \emph{spine}. The number of vertices on the spine is finite and the number of vertices of type $\circ$ on the spine  is equal to the number of vertices of type $\bullet$. The graph $\cc_{\aa}$ is then formed by attaching Farey trees to all vertices of type $\bullet$ on the spine so that they should expand both \emph{inside} and \emph{outside} the spine. Each so attached Farey tree is called a Farey branch. The number of consecutive Farey branches that point in the same direction is called a Farey bunch. The graph $\cc_{\aa}$ and hence the conjugacy class of $W_{f_{\aa}}$ is completely determined by the formation of these Farey bunches and the number of Farey branches within these bunches. 

\begin{example}
	For $K = \QQ(\sqrt{30})$, we set $\aa = (2,\sqrt{30})$. $N(\aa) = 2$. This gives rise to the indefinite binary quadratic form $f_{\aa}(X,Y) = 2X^{2} -15Y^{2}$. We have $W_{f_{\aa}} = \begin{pmatrix}
	11&30\\4&11
	\end{pmatrix}$ and in terms of the generators one has $W_{f_{\aa}} = (LS)^{2}L^{2}S (LS)^{2} L^{2}S (LS)^{2}$. Figure~\ref{fig:cark/example} depicts the corresponding \c{c}ark.
	\label{ex:sqrt/30}
\end{example}

\begin{figure}
	\centering
	\includegraphics[scale=0.7]{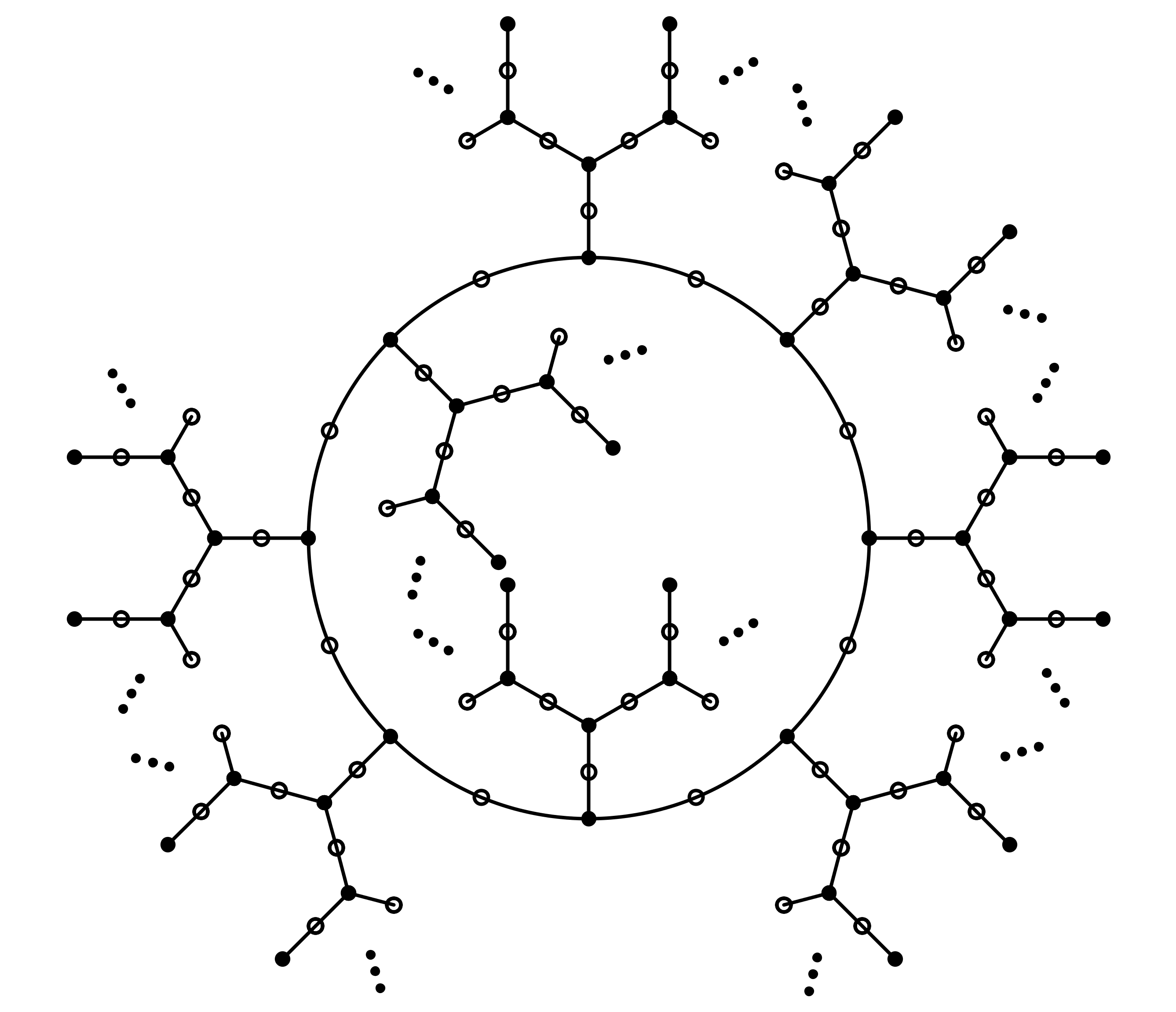}
	\caption{The \c{c}ark representing $\aa = (2,\sqrt{30})$.}
	\label{fig:cark/example}
\end{figure}

\subsection{\c{C}arks as integral points.}

Let $\aa$ be a narrow ideal class in $K$ and $W_{f_{\aa}} \in \psl$ be its fundamental automorphism. Elements in the conjugacy class $[W_{f_{\aa}}]$ can be partially ordered according to their lengths (i.e. number of letters $S,L$ and $L^{2}$ that appear). Under the correspondence between the edges of the \c{c}ark $\cc_{\aa}$ and $[W_{f_{\aa}}]$ one may observe that those that are of smallest length (called \emph{minimal words}) correspond exactly to edges on the spine. Minimal words are not unique because if $W$ is such a word then so is $SWS$. In fact, minimal words can be written as a disjoint union of those that start with $S$ and those that start either with $L$ or with $L^{2}$. Among the latter there are those that can be written of the form $(LS)^{m_{1}}(L^{2}S)^{n_{1}} \ldots (LS)^{m_{k}}(L^{2}S)^{n_{k}}$; where $m_{1}, n_{1},\ldots,m_{k},n_{k}\geq 1$. To each such element in the conjugacy class, we associate the ordered pair $(m_{1},n_{1},\ldots,m_{k},n_{k})$ and call $k$ the length of the \c{c}ark. Observe that the integers $m_{i}$ represent the number of Farey trees in consecutive Farey bunches that expand in the direction of the outer boundary. Analogously $n_{i}$ stand for the number of Farey tree in the Farey bunches that expand in the direction of the inner boundary.

A couple of remarks are in order. If the conjugacy class of a word $W$ gives rise to the sequence $(m_{1},n_{1},\ldots,m_{k},n_{k})$, then $W^{l}$ gives rise to the same sequence repeated $l$-times, in particular it is represented by a $2kl$-tuple. Let us define a \c{c}ark to be \emph{primitive} if it is not a repetition of a shorter \c{c}ark. Therefore, although $W$ and $W^{l}$ give rise to the same binary quadratic form their \c{c}arks are different. Secondly, the number $k_{\aa} := k$ is fixed for a narrow ideal class $\aa$ and the length of the minimal word is equal to $2\sum_{i=1}^{k}(m_{i}+n_{i})$. However, different narrow ideal classes of the same number field $K$ may be represented by minimal words of different lengths, e.g. for $K = \QQ(\sqrt{30})$ $H^{+}(K) \cong (\ZZ/2\ZZ)^{2}$, with two narrow ideal classes that lies above the class of the principal ideal being represented by a 2-tuple, and as we have seen earlier(Example~\ref{ex:sqrt/30}), the two narrow ideal classes lying above $[(2,\sqrt{30})]$ being represented by a 4-tuple. We are now ready to prove our main theorem:

\begin{theorem}
	Let $K$ be a real quadratic number field of discriminant $\Delta$. Then each narrow ideal class in $K$ gives rise to an integral solution of the equation 
	\begin{eqnarray}
		\mathcal{L}_{2k}(x_{1},y_{1},\ldots,x_{k},y_{k})^{2} - 4 = z^{2}\Delta;
		\label{eq:cark/hypersurface}
	\end{eqnarray}
	for some positive integer $k$ and for some $z \in \ZZ$ which depends only on $K$.
	\label{thm:affine/equation}
\end{theorem}
\begin{proof}
	We let $k = k_{K}$ to be the maximum of $k_{\aa}$ as $\aa$ runs through $H^{+}(K)$ and let $\Delta$ be the discriminant of $K$. Set $(x_{o},z_{o})$ to be the fundamental solution of the corresponding Pell equation $X^{2} - \Delta Z^{2} = 4$. Remark that for any narrow ideal class, the fundamental automorphism will be obtained using $(x_{o},z_{o})$. For each \c{c}ark represented by $2l$-tuple, say $(m_{\aa,1},n_{\aa,1}, \ldots m_{\aa,l},n_{\aa,l})$, for $l<k$ we complete it to a $2k$-tuple by appending $2(k-l)$-many zeroes to the end of the tuple and obtain $(m_{\aa,1},n_{\aa,1}, \ldots m_{\aa,l},n_{\aa,l},0,\ldots,0)$. Given such a $2k$-tuple, say $(m_{1},n_{1},\ldots,m_{k}n_{k})$ we set:
	$$W = (LS)^{m_{1}}(L^{2}S)^{n_{1}}\ldots (LS)^{m_{k}}(L^{2}S)^{n_{k}}.$$
	By construction the matrix $W \in \psl$ is a fundamental automorphism of the binary quadratic form $f_{\aa}$. Using the correspondence between automorphisms and solutions of the Pell equation $X^{2}+\Delta Z^{2} = 4$, see Section~\ref{sec:automorphisms}, we obtain $\tr(W) = x = x_{o}$ and satisfies $x_{o}^{2} - 4 = \Delta z_{o}^{2}$.
	
	Now, we observe that $M(m,n) = (LS)^{m}(L^{2}S)^{n}$. The multivariate Lucas polynomial $\mathcal{L}_{2k} = 2A_{k}$ is merely the trace of the matrix $M(x_{1},y_{1},\ldots,x_{k},y_{k})$ hence trace of $W$ is equal to $\mathcal{L}_{2k}(m_{1},n_{1},\ldots,m_{k},n_{k})$.
\end{proof}

\begin{definition}
	We let $C_{K}\subset \CC^{2k_{K}}$ denote the solution set of the equation 
	\begin{eqnarray*}
		\mathcal{L}_{2k_{K}}(x_{1},y_{1},\ldots,x_{k_{K}},y_{k_{K}})^{2} - 4 = z^{2}\Delta;
	\end{eqnarray*}
	refer to it as the \emph{affine \c{c}ark hypersurface}.
\end{definition}

Recall that there is an action of $\ZZ/k\ZZ$ on $\mathcal{L}_{2k}$. This means that the set $C_{K}$ admits an action of $\ZZ/k_{K}\ZZ$, see Equation~\ref{eq:lucas/invariance}. 

\begin{definition}
	The \emph{affine \c{c}ark surface} of $K$ is defined as the quotient $C_{K}/(\ZZ/k_{K}\ZZ)$. This surface will be denoted by $\mathcal{C}_{K}$. 
\end{definition}

\begin{corollary}
	Let $K$ be a real quadratic number field. Then there is a one to one correspondence between integral points of the \c{c}ark surface $\mathcal{C}_{K}$ and narrow ideal classes in $K$.
	\label{thm:quotient}
\end{corollary}

\begin{proof}
	One way of this correspondence can be obtained by using Theorem~\ref{thm:affine/equation} and noting the fact that the action of $\ZZ/k_{K}\ZZ$ does not change the conjugacy class of the fundamental automorphism of narrow ideal classes. 
	
	Conversely, if we start with an integral point on $\mathcal{C}_{K}$, say $(m_{1},n_{1},\ldots,m_{k_{K}},m_{k_{K}})$, one can construct the element 
	$$W = (LS)^{m_{1}}(L^{2}S)^{n_{1}}\ldots (LS)^{m_{k}}(L^{2}S)^{n_{k}}$$
	and look at the narrow ideal class that arises from the binary quadratic form whose fundamental automorphism is $W$, namely if 
	$W = \begin{pmatrix}
				p&q\\r&s
	     \end{pmatrix}
	$ then the corresponding indefinite primitive binary quadratic form is $f(X,Y) = \frac{1}{\delta} \left(rX^{2} +(s-p)XY - qY^{2}\right)$; where $\delta$ is the greatest common divisor of $r$, $s-p$ and $q$. Since this point is a solution of Equation~\ref{eq:cark/hypersurface}; where by definition $z_{o}$ is minimal, the discriminant of this form must be square-free.
\end{proof}

Let us conclude the paper with a few remarks. Instead of considering the \emph{fundamental solution} we may consider other $z$ arising from non-fundamental solutions of the corresponding Pell equation. Even more generally, we may treat $z$ as variable in the equation and consider the affine hypersurface with equation
$$\mathcal{L}_{2k}(x_{1},y_{1},\ldots,x_{k},y_{k})^{2} - 4 = z^{2}\Delta$$ in $\CC^{2k+1}$. Each integral point on the quotient of this hypersurface with the obvious action of $\ZZ/k\ZZ$ gives rise to a binary quadratic form whose discriminant's square-free part is equal to $\Delta$. Such integral points gives rise to non-maximal orders in $\ZZ_{K}$. One may then intersect the hypersurface with $z = \lambda$ planes, where $\lambda \in \ZZ$, and then consider integral points of the intersection. In this case, again each integral point gives rise to a narrow ideal class in an appropriate class group, see \cite[Theorem~5.2.9]{computational/nt/cohen}. One may generalize Corollary~\ref{thm:quotient} immediately to this case. Indeed, assuming one can compute the class number for $K$, or almost equivalently find the number of integral points on the \c{c}ark surface of $K$, one can determine the number of integral points of this surface, see \cite[Corollary~7.28]{cox/primes/of/the/form}. 



\paragraph{Acknowledgments.} This research is supported by T\"{U}B\.{I}TAK 1001 Grant 114R073.

\bibliographystyle{abbrv}

\end{document}